\documentclass[aip,jmp,sd,amsmath,amssymb,preprint,author-numerical]{revtex4-1}
\usepackage{dcolumn}
\usepackage{bm}
\usepackage{comment}

\newtheorem{theorem}{Theorem}

\begin{document}

\title{Group classification of a class of generalized nonlinear Kolmogorov equations\\ and exact solutions}

\author{I. Rassokha}
 \email{innaolha@gmail.com}
 \author{M. Serov}
 \email{mserov4@gmail.com}
\affiliation{Poltava National Technical Yuri Kondratyuk University,
24 Pershotravnevyi Ave., 36011 Poltava, Ukraine}%

\author{S. Spichak}
 \email{spichak@imath.kiev.ua}
\affiliation{Institute of Mathematics, National Academy of Sciences of Ukraine, 3 Tereshchenkivs'ka Str., 01004 Kyiv-4, Ukraine}

\author{V. Stogniy}
 \email{stogniyvaleriy@gmail.com}
\affiliation{National Technical University of Ukraine ``Igor Sikorsky Kyiv Polytechnic Institute'',
37 Peremohy Ave., 03056 Kyiv, Ukraine}

\date{\today}

\begin{abstract}
A class of generalized nonlinear Kolmogorov equations is investigated. We present the group classification
of Lie symmetries of the class with respect to the group of equivalence transformations.
We find a number of exact solutions of nonlinear
Kolmogorov equations which has the maximal symmetry properties.
\end{abstract}

\pacs{02.20.Sv, 02.30.Jr, 92.60.Ek, 94.10.Lf}
\keywords{generalized nonlinear Kolmogorov equation, group classification, exact solution}

\maketitle

\section{\label{1}Introduction}

In this article, we study the class of the generalized nonlinear Kolmogorov type equations:
\begin{eqnarray}\label{spich1}
u_t -u_{xx} +f^1(u)u_y =f^2(u),
\end{eqnarray}
where $u=u(t,x)$, $u_t=\frac{\partial u}{\partial t}$, $u_x=\frac{\partial u}{\partial x}$, $u_x=\frac{\partial u}{\partial x}$, $u_{xx}=\frac{\partial^2 u}{\partial x^2}$; $f^1(u)$, $f^2(u)$ are arbitrary smooth functions of their variables, and $f^1\ne \text{const}$. If $f^1=\text{const}$ then equation (\ref{spich1}) can be reduced to
a spatially one-dimensional nonlinear one ($f^1=0$) by corresponding Galilean transformation.
The group classification of such equations was carried out, in particular, in Ref.~\onlinecite{stog28}.

Equation (\ref{spich1}) occurs in problems of financial mathematics and in studying of the physical
nonlinear phenomena such as the combined effects of diffusion and convection of matter \cite{stog1}.

Wide application of equation (\ref{spich1}) causes an undeniable interest in obtaining solutions.
One of the most powerful methods for constructing exact solutions of nonlinear PDEs is the classical Lie method \cite{stog2, stog3, stog4},
and its various generalizations and modifications \cite{stog5}.

A common method for solving nonlinear differential equations is the Lie method, which is based on the principle of symmetry.
S.~Lie was the first who used an invariance algebra of differential equations for symmetry reduction and for finding exact solutions.
The theoretical-group methods allow you to integrate differential equations, which have non-trivial invariance group.
That's why the urgent task is a complete group classification of differential equations with arbitrary functions that allows
to select the equations with broad symmetry properties from the given class of equations.

The interest in the group classification of equations exists for
more than 50 years, since L.V.~Ovsyannikov in his
article \cite{stog6} has made the current formulation of the problem
of group classification of differential equations, proposed a method
for its solution and implemented a group classification of nonlinear
heat equation. Detailed review of publications of group
classification of differential equations until the mid 90-ies of the
last century is given in Ref.~\onlinecite{stog7}. Group
classification of evolution equations in spaces of dimension higher
than the two-dimensional space-time is considered in Refs.~\onlinecite{stog8, stog9, stog10, stog11, stog12, stog13, stog15, stog17}.

The problem of group classification of equation (\ref{spich1}) for the case $f^2(u)=0$ was carried out in Refs.~\onlinecite{stog15, stog18} relatively to an arbitrary function $f^1(u)$, and, in particular, shows that the nonlinear Kolmogorov equation
\begin{eqnarray*}
u_t -u_{xx} -uu_y =0
\end{eqnarray*}
admits $6$-parameter group of local transformations and using the method of symmetry reduction, built exact solutions of that equation.

Thus, it becomes apparent urgency of the task complete group classification of equation \eqref{spich1} with respect to arbitrary functions and which was resolved by us.

\section{\label{2}The basic algebra $\boldsymbol{A^{\ker}}$}

Consider a one-parameter Lie group of local transformations in a space of variables $(t,x,y,u)$ with an infinitesimal operator
\begin{eqnarray}\label{spich2}
G=\xi ^0(t,x,y,u)\partial _t +\xi ^1(t,x,y,u)\partial _x
 +\xi^2(t,x,y,u)\partial _y +\eta (t,x,y,u)\partial _u ,
\end{eqnarray}
which keeps equations (\ref{spich1}) invariant. The Lie criterion of infinitesimal invariance yields the
following determining equations for functions $\xi^0$, $\xi^1$, $\xi^2$, $\eta$ and also for arbitrary
elements $f^1(u)$ and $f^2(u)$:
\begin{eqnarray}
&\label{spich3}\xi _x^0 =\xi _u^0 =\xi _u^1 =\xi _x^2 =\xi _u^2 =\eta _{uu} =0,\\
&\label{spich4}\xi _t^0 -2\xi _x^1 +\xi _y^0 f^1=0,\\
&\label{spich5}\xi _t^1 -\xi _{xx}^1 +\xi _y^1 f^1+2\eta _{ux} =0,\\
&\label{spich6}\eta \dot f^1 =\big(\xi _y^2 -2\xi _x^1 \big)f^1+\xi _t^2 ,\\
&\label{spich7}\eta \dot{f}^2 =\big(\eta _u -2\xi _x^1 \big)f^2+\eta _y f^1+\eta _t -\eta _{xx} ,
\end{eqnarray}
where the dot above the functions $f^1$, $f^2$ denote differentiation with respect to~$u$.

Equations (\ref{spich3}) do not contain arbitrary elements. Integration of them yields
\begin{eqnarray}\label{spich8}
\xi^0=\xi^0(t,y),\quad \xi^1=\xi^1(t,x,y),\quad
\xi^2=\xi^2(t,y),\quad  \eta =\alpha (t,x,y)u+\beta (t,x,y),
\end{eqnarray}
where $\alpha$, $\beta$ are arbitrary smooth functions.

Since the functions $\xi^0$, $\xi^1$ and $\alpha$ do not on variable $u$ and $f^1\ne \mbox{const}$ we obtain from (\ref{spich4}),
(\ref{spich5}), (\ref{spich8})
\begin{eqnarray}\label{spich9}
\xi_y^0 =0,\quad \xi _t^0 -2\xi _x^1 =0,\quad \xi _y^1 =0,\quad \xi _t^1 =-2\alpha _x .
\end{eqnarray}
Since (\ref{spich8}) the solutions of equations (\ref{spich9}) have the form
\begin{eqnarray}\label{spich10}
\xi ^0=\gamma (t),\quad \xi ^1=\tfrac{1}{2}{\gamma }'(t)x+\delta(t),\quad
\alpha =-\tfrac{1}{8}{\gamma }''(t)x^2-\tfrac{1}{2}{\delta }'(t)x+\mu(t,y),
\end{eqnarray}
where $\gamma$, $\delta$, $\mu$ are arbitrary smooth functions.

Taking into account (\ref{spich8}), (\ref{spich10}) and differentiating equation (\ref{spich6}) on $x$ we have\\
$(\alpha _xu +\beta _x )\dot f^1 =0$,
that is $\alpha _x =\beta _x =0$. Then relations (\ref{spich10}) has the form
\begin{eqnarray}\label{spich11}
\xi^0=2c_1 t+c_0 ,\quad \xi^1=c_1 x+c_2,\quad
\xi^2=\xi^2(t,y),\quad \eta =\alpha (t,y)u+\beta (t,y),
\end{eqnarray}
where $c_0 ,\ c_1 ,\ c_2 $ are constants.
So, equations (\ref{spich6}), (\ref{spich7}) has the form
\begin{eqnarray}
&\label{spich12}(\alpha u+\beta )\dot f^1 =(\xi _y^2 -2c_1 )f^1+\xi _t^2,\\
&(\alpha u+\beta )\dot f^2 =(\alpha -2c_1 )f^2+(\alpha _y u+\beta _y )f^1+\alpha _t u+\beta _t .\label{spich13}
\end{eqnarray}
Splitting system (\ref{spich12}), (\ref{spich13}) with respect to the arbitrary elements and their
non-vanishing derivatives gives the ``trivial'' solution (\ref{spich11})
\begin{eqnarray}\label{spich14}
\xi ^0=c_0 ,\quad \xi ^1=c_2 ,\quad \xi ^2=c_3 ,\quad \alpha =\beta =0,
\end{eqnarray}
which corresponds to the coefficients of the operators (\ref{spich2}) from the basic Lie
algebra $A^{\ker }$ of (\ref{spich1}). As a result, the following theorem is true.

\begin{theorem}\label{T1}
For arbitrary functions $f^1$ and $f^2$ the Lie symmetry algebra of the equation \eqref{spich1} is
$A^{\ker }=\left\langle {\partial _t ,\partial _x ,\partial _y }\right\rangle$.
\end{theorem}
The problem of group classification consists in finding of all possible inequivalent cases when the solution
of system (\ref{spich12}), (\ref{spich13}) leads to Lie algebra $A^{\max }$, that satisfies the condition $A^{\ker}\subset A^{\max }$,
i.e., to four- and higher-dimensional Lie algebra.

\section{\label{3}Group of equivalence transformations}

The next step of the algorithm of group classification is finding equivalence transformations \cite{stog2, stog19, stog20} of class (\ref{spich1}).

\begin{theorem}\label{T2}
An arbitrary nonlinear Kolmogorov type equation of the form \eqref{spich1} can be mapped to another equation of the same form
\begin{eqnarray}\label{spich15}
\bar {u}_{\bar {t}} -\bar {u}_{\bar {x}\bar {x}} +F^1(\bar {u})\bar{u}_{\bar {y}} =F^2(\bar {u})
\end{eqnarray}
by the local transformations
\begin{eqnarray*}
 \bar {t}=T(t,x,y,u),\quad \bar {x}=X(t,x,y,u),\quad \bar{y}=Y(t,x,y,u), \quad \bar {u}=U(t,x,y,u),
\end{eqnarray*}
with the correctly-specified smooth functions $T$, $X$, $Y$ and $U$ and Jacobian
$J(t,x,y,u)=\frac{\partial(T,X,Y,U)}{\partial(t,x,y,u)}\ne 0$,
if these functions are of the form
\begin{eqnarray}\label{spich16}
\bar {t}=a_1^2 t+a_2,\quad \bar {x}=a_1 x+a_3 ,\quad \bar{y}=Y(t,y), \quad \bar {u}=\theta (t,y) u+\nu (t,y) ,
\end{eqnarray}
and following equalities take place:
\begin{eqnarray}
&\label{spich17}Y_t +Y_y f^1(u)=a_1^2 F^1(\bar {u}),\\
&\label{spich18}\theta _t u+\nu _t +(\theta _y u+\nu _y )f^1(u)+\theta f^2(u)=a_1^2 F^2(\bar{u}),
\end{eqnarray}
where $a_1$, $a_2$, $a_3 $ are real arbitrary constants, $a_1 \ne 0$, $Y_y \ne 0$, $\theta (t,y)\ne 0$.
\end{theorem}
From (\ref{spich16})--(\ref{spich18}) we obtain the set of equivalence transformations which is the local transformations
group $G^{\rm equiv}$ preserving the form equations (\ref{spich1}).

\begin{theorem}\label{T3}
Group $G^{\rm equiv}$ of the equation \eqref{spich1} contains the following continuous transformations:
\begin{eqnarray}\label{spich19}
\begin{aligned}
&\hspace{-5mm} \bar{t}=a_1^2 t+a_2 ,\ \bar {x}=a_1 x+a_3 ,\ \bar{y}=a_4 y+a_5 t+a_{6},\ \bar{u}=a_7 u+a_8,\\
&\hspace{-5mm} F^1=\frac{a_4 }{a_1^2 }f^1+\frac{a_5 }{a_1^2 },\ F^2=\frac{a_7 }{a_1^2 }f^2,
\end{aligned}
\end{eqnarray}
where $a_i$, $i\in\{1,\ldots,8\}$  are arbitrary constants, $a_1 a_4 a_7 \ne 0.$
\end{theorem}

Note that the transformations (\ref{spich19}) reduce equations (\ref{spich1}) to (\ref{spich15}) for
arbitrary functions $f^1$, $f^2$. For special view of these functions there are another (additional) equivalence
transformations. These transformations will be further applied to unite the equations~(\ref{spich1}) with the same Lie symmetry.

\section{\label{4}Classification of Lie symmetries}

The next step of the algorithm of group classification is finding a complete set of inequivalent equations (\ref{spich1}) with respect
to the transformations from $G^{\rm equiv}$ which are invariant under Lie algebra $A^{\max }\supset A^{\ker}$ ($\dim A^{\max}>\dim A^{\ker}$).
Hereinafter, the notation $A^{\max}$ implies the fulfillment of this condition.

\begin{center}
{\bf Classification of functions $f^1$}
\end{center}

To construct all possible form of the functions $f^1$ for which corresponding set of solutions of system (\ref{spich12}), (\ref{spich13})
is wider than ``trivial'' solution (\ref{spich14}) we analyze these equations along zero and non-zero values of the $\alpha$ and $\beta$.

Consider the simplest {\bf Case} when for any solution (\ref{spich11}) $\alpha =0,\quad \beta =0$, i.e., $\eta =0$.
Then system (\ref{spich12}), (\ref{spich13}) is reduced to the algebraic one with respect to~$f^1$ and~$f^2$
\begin{eqnarray}\label{spich20}
(\xi _y^2 -2c_1 )f^1+\xi _t^2 =0,\quad c_1 f^2=0.
\end{eqnarray}
One easily sees that Lie algebra $A^{\max }$ occurs only under the condition $c_1 \ne 0,$ otherwise the general solution of (\ref{spich20}) leads to (\ref{spich14}) (corresponding to the algebra~$A^{\ker}$).

The condition $c_1 \ne 0$ immediately gives $f^2\equiv 0$. Therefore, system~(\ref{spich12}), (\ref{spich13}) with arbitrary functions
$f^1\ne \mbox{const}$ and~$f^2\equiv 0$ has the general solution
\begin{eqnarray}\label{spich21}
\hspace{-3mm}\xi ^0=2c_1 t+c_0 ,\ \xi ^1=c_1 x+c_2 ,\ \xi^2=2c_1 y+c_3,
\end{eqnarray}
where $c_i$ $(i\in\{0,1,2,3\})$ are arbitrary constants. The solution (\ref{spich21}) leads to Lie algebra
\begin{eqnarray*}
\left\langle {\partial _t ,\ \partial _x ,\ \partial _y ,\ 2t\partial _t +x\partial _x +2y\partial _y } \right\rangle .
\end{eqnarray*}

Let us consider the condition $\alpha ^2+ \beta ^2\ne 0$. Since the functions $\xi ^2$, $\alpha $ and~$\beta $ do not on variable~$u$, one can obtain possible view the function$f^1$ from ODE~(\ref{spich12}).

\begin{theorem}\label{T4}
If an equation of the form \eqref{spich1} admits algebra $A^{\max }$ which includes infinitesimal operator \eqref{spich2} with $\eta \ne 0$ then the
function $f^1$ is equivalents with respect to the transformations $G^{\rm equiv}$ one of the form: $u^k$ $(k\ne 0)$, $e^u$, $\ln u$.
\end{theorem}

\begin{table*}
\caption{\label{tab:1}Classification of functions $f^1$}
\begin{ruledtabular}
\begin{tabular}{p{20pt}p{100pt}p{150pt}p{100pt}}
No. & $f^1(u)$ &$\begin{array}{l} \text{Equivalence}\\ \text{transformations}\end{array}$ & $F^1(\bar {u})$ \\
\hline
1&$k_1 e^{k_2 u}+k_3$ ,
 &$ \bar {t}=t$,  $\bar {x}=x$, $\bar {u}=k_2 u$
 &$e^{\bar {u}}$\\
 &$ k_1 \ne 0$, $k_2 \ne 0$&$\bar {y}=\frac{1}{k_1 }y-\frac{k_3 }{k_1 }t$,&\\
2&$k_1 \ln (u+k_2 )+k_3$ ,
 &$ \bar {t}=t$, $\bar {x}=x$, $\bar {u}=u+k_2$
&$\ln \bar {u}$ \\
 &$ k_1 \ne 0$&$\bar {y}=\frac{1}{k_1 }y-\frac{k_3 }{k_1 }t$,&\\
3&$(k_1 u+k_2 )^m+k_3$ ,
&$\bar {t}=t$,\ $\bar {y}=y-k_3 t$,
&$\bar {u}^m$, \ $m\ne 0$\\
&$m\ne 0$,\ $k_1 \ne 0$
&$\bar {x}=x$,\ $\bar {u}=k_1 u+k_2$
&\\
\end{tabular}
\end{ruledtabular}
\end{table*}

{\bf Proof.} \textbf{Case (1)} $\alpha =0$, $\beta \ne 0$. Let us consider values $t$, $x$, $y$ as some parameters for equation (\ref{spich12}).
The general solution of equation (\ref{spich12}) has the form
\begin{eqnarray*}
\begin{aligned}
&f^1=C\exp{\left( {\frac{\xi _y^2 -2c_1 }{\beta }} u\right)}-\frac{\xi _t^2}{\xi _y^2 -2c_1 }\quad \text{if}\quad
\xi _y^2 -2c_1 \ne 0,\\
&\text{or}\quad f^1=\frac{\xi _y^2}{\beta }u+C\quad \text{if} \quad \xi _y^2 -2c_1 =0,
\end{aligned}
\end{eqnarray*}
where $C=C(t,x,y)$ is an arbitrary function. Because the function $f^1$ is dependent on $u$ only, it will be of the form
\begin{eqnarray*}
f^1=k_1 u+k_2 \quad (k_1 \ne 0), \quad \text{or}\quad
f^1=k_1 e^{k_2 u}+k_3 \quad (k_1 ,k_2 \ne 0),
\end{eqnarray*}
where $k_i$ $(i\in\{1,2,3\})$ are arbitrary constants.

\textbf{Case (2)} $\alpha \ne 0$. The general solution of equation (\ref{spich12}) has the form
\begin{eqnarray*}
\begin{aligned}
&f^1=C(\alpha u+\beta )^{\frac{\xi _y^2 -2c_1 }{\alpha }}-\frac{\xi _t^2 }{\xi_y^2 -2c_1 }
\quad \text{if} \quad \xi _y^2 -2c_1 \ne 0,\\
&\text{or}\quad f^1=\frac{\xi _t^2}{\alpha }\ln \left| {\alpha u+\beta } \right|+C\quad \text{if} \quad \xi _y^2 -2c_1 =0,
\end{aligned}
\end{eqnarray*}
where $C=C(t,x,y)$ are arbitrary function. As the function $f^1$ is dependent on~$u$ only, it will be of the form
\begin{eqnarray*}
f^1=(k_1 u+k_2 )^m+k_3,\quad \text{or}\quad f^1=k_1 \ln (u+k_2)+k_3 \quad (m,k_1 \ne 0),
\end{eqnarray*}
where $m$, $k_i$ $(i\in\{1,2,3\})$ are arbitrary constants.

Combining all obtained in the Cases 1, 2 forms of functions $f^1$ and using the set of equivalence transformations~(\ref{spich19})
we get simplified view $F^1$ of these functions (see Table~\ref{tab:1}). $\blacksquare$

\begin{center}
{\bf Classification of functions $f^2$}
\end{center}

For each simplified functions $f^1$ from Table~\ref{tab:1} we find all possible species of the function $f^2$
(the necessary conditions) under solving classifying equation~(\ref{spich13}). As before we impose the requirement that
the equation has solutions different from ``trivial'' one~(\ref{spich14}) which correspond to the algebra $A^{\ker}$.

\textbf{Case (1)} $f^1=e^u$. Substituting function $f^1=e^u$ in equation (\ref{spich12}) we obtain the relations
$\alpha =0$,\quad $\beta =\xi_y^2 -2c_1$ ,\quad $\xi _t^2 =0$.
So from relations (\ref{spich11}) we find that
\begin{eqnarray*}
\xi ^0=2c_1 t+c_0 ,\quad \xi^1=c_1 x+c_2 , \quad \xi^2=\xi^2(y),\quad \beta =\xi _y^2-2c_1.
\end{eqnarray*}

Then classifying equation (\ref{spich13}) has the form
\begin{eqnarray}\label{spich22}
(\xi_y^2-2c_1 )\dot{f}^2=-2c_1 f^2+\xi_{yy}^2 e^u.
\end{eqnarray}

1) Let $\xi_{yy}^2 \ne 0$. Differentiating both sides of equation (\ref{spich22}) by~$y$, we obtain that
$\xi_{yy}^2 \dot{f}^2=\xi _{yyy}^2 e^u$. Hence, the general solution has the form
\begin{eqnarray}\label{spich23}
f^2=k_1 e^u+k_2 ,
\end{eqnarray}
where $k_1$ and $k_2$ are arbitrary constants.

2) Let $\xi_{yy}^2 = 0$, i.e., $\xi^2=c_3 y+c_4$. Then equation (\ref{spich22}) has the form
\begin{eqnarray}\label{spich24}
(c_3-2c_1 )\dot{f}^2=-2c_1 f^2.
\end{eqnarray}
If $c_3-2c_1 =0$ then $c_1 \ne 0$ (otherwise Lie algebra coincides $A^{\ker}$) and $f^2 =0$, that falls into the previous \textbf{Case (1)} ($k_1 = k_2 = 0$). If $c_3-2c_1 \ne 0$ then general solution of equation (\ref{spich24}) has the form
\begin{eqnarray}\label{spich25}
f^2=k_1 e^{mu} ,
\end{eqnarray}
where $k_1$ and $m$ are arbitrary constants.

\textbf{Case (2)} $f^1=\ln u$. Substituting function $f^1=\ln u$ in equation (\ref{spich12}) we obtain the relations
$\alpha =\xi_t^2$,\quad $\beta =0$,\quad $\xi_y^2 =2c_1$.
So from relations (\ref{spich11}) we find that
\begin{eqnarray*}
\xi^0=2c_1 t+c_0 ,\quad \xi^1=c_1 x+c_2,\quad \xi^2=2c_1 y+q(t),\quad \alpha =q'(t).
\end{eqnarray*}
Then classifying equation (\ref{spich13}) has the form
\begin{eqnarray}\label{spich26}
q'(t)u\dot{f}^2=(q'(t)-2c_1)f^2+q''(t)u.
\end{eqnarray}

1) Let $q''(t) \ne 0$. Differentiating both sides of equation (\ref{spich26}) by $t$ and $u$, we have $\ddot{f}^2=q'''(q''u)^{-1}$.
Hence, the general solution has the form $f^2=k_1u\ln u+k_2u+k_3$, where $k_1$, $k_2$ and $k_3$ are arbitrary constants.
Substituting that solution into equation (\ref{spich26}), we obtain that $k_3=0$. So,
\begin{eqnarray}\label{spich27}
f^2=k_1u\ln u+k_2u.
\end{eqnarray}

2) Let $q''(t) = 0$. If $q'(t)\ne 0$ then from (\ref{spich26}) we have
\begin{eqnarray}\label{spich28}
f^2=k_1u^m,
\end{eqnarray}
where $k_1$ are an arbitrary constant. If $q'(t)= 0$ then from (\ref{spich26}) we have $f^2=0$ (because $c_1\ne 0$) that falls into case~(\ref{spich28})
when $k_1=0$.

\textbf{Case (3)} $f^1=u^m$, $m\ne 0, 1$. Substituting function $f^1=u^m$ in equation (\ref{spich12}) we obtain the relations
$\beta =\xi_t^2=0$,\quad $m\alpha =\xi_y^2 -2c_1$.
 So from relations (\ref{spich11}) we find that
$$ \xi ^0=2c_1 t+c_0 ,\quad \xi^1=c_1 x+c_2 ,\quad \xi^2=\xi^2(y),\quad \alpha =\frac{\xi _y^2-2c_1}{m}.$$
Then classifying equation (\ref{spich13}) has the form
\begin{eqnarray}\label{spich29}
 \frac{\xi _y^2-2c_1}{m}u\dot{f}^2=\left(\frac{\xi _y^2-2c_1}{m}-2c_1\right)f^2+\frac{\xi_{yy}^2}{m}u^{m+1}.
\end{eqnarray}

1) Let $\xi_{yy}^2 \ne 0$. Differentiating both sides of equation (\ref{spich29}) by $y$, we obtain that
\begin{eqnarray}\label{spich30}
 u\dot{f}^2=f^2+\frac{\xi_{yyy}^2}{\xi_{yy}^2}u^{m+1}.
\end{eqnarray}
Then the general solution of equation (\ref{spich30}) has the form
\begin{eqnarray}\label{spich31}
f^2=k_1u^{m+1}+k_2u,
\end{eqnarray}
where $k_1$ and $k_2$ are arbitrary constants.

2) Let $\xi_{yy}^2 = 0$, i.e. $\xi^2 = c_3y+c_4$. Then equation (\ref{spich29}) has the form
\begin{eqnarray}\label{spich32}
 \frac{c_3-2c_1}{m}u\dot{f}^2=\left(\frac{c_3-2c_1}{m}-2c_1\right)f^2.
\end{eqnarray}
If $c_3-2c_1=0$ then $c_1\ne 0$ (otherwise Lie algebra coincides with $A^{\ker}$) and $f^2=0$, that falls into the previous case~(\ref{spich31}).
If $c_3-2c_1 \ne 0$ then general solution of equation (\ref{spich32}) has the form
\begin{eqnarray}\label{spich33}
f^2=k_1u^n.
\end{eqnarray}

\textbf{Case (4)} $f^1=u$. Substituting function $f^1=u$ in equation (\ref{spich12}) and taking account (\ref{spich11}) we find relations
\begin{eqnarray}\label{spich34}
\xi ^0=2c_1 t+c_0 ,\quad \xi^1=c_1 x+c_2 ,\quad \xi^2=\xi^2(t,y),\quad \alpha =\xi _y^2-2c_1,\quad \beta =\xi^2_t.
\end{eqnarray}

1) Let $\alpha =\beta = 0$. In that case $c_1\ne 0$ (otherwise Lie algebra coincides with $A^{\ker}$). Then it is following from (\ref{spich13})
that $f^2=0$ what falls into the case (\ref{spich21}).

2) Let $\alpha = 0$, $\beta \ne 0$. From (\ref{spich34}) we obtain one of the possible general form of function $f^2$ for equation (\ref{spich13})
\begin{eqnarray*}
& f^2=\frac{\beta_y}{2\beta}u^2+ \frac{\beta_t-\beta_{xx}}{\beta}u+C \quad (c_1=0),\\
&f^2=C\exp \left(-\frac{2c_1}{\beta}u\right)+ \frac{\beta_y}{2c_1}u+\frac{2c_1\beta_t-\beta\beta_y}{4c_1^2} \quad (c_1\ne 0),
\end{eqnarray*}
where $C(t,x,y)$ are an arbitrary function.

As the function $f^2$ is dependent on $u$ only, it will be of the form
$$f^2=k_1u^2+ k_2u+k_3 \quad \text{or} \quad f^2=k_1\exp(k_2 u)+ k_3u+k_4 \quad (k_i\ne 0, i\in\{1,2\}),$$
where $k_i$ $(i\in\{1,2,3,4\})$ are arbitrary constants.

If substitute $f^2=k_1u^2+ k_2u+k_3$ when $k_1\ne 0$ into (\ref{spich13}), then the system has the solution which leads to the algebra $A^{\ker}$.
Also the function $f^2=k_1\exp(k_2 u)+ k_3u+k_4$ $(k_i\ne 0, i\in\{1,2\})$, $k_3^2+k_4^2\ne 0$ leads to the algebra $A^{\ker}$
for equation~(\ref{spich1}). So in this case the possible form of the function under consideration has the view{\samepage
\begin{eqnarray}\label{spich35}
f^2=k_1u+k_2, \quad f^2=k_1\exp(k_2 u),
\end{eqnarray}
where $k_i$ $(i\in\{1,2\})$ are arbitrary constants.}

3) Let $\alpha\ne 0$. Similarly as in the case 1) the analysis of equation (\ref{spich13}) leads to the following possible forms of the function $f^2$:
\begin{eqnarray}
&\label{spich36}f^2=k_3u^i\ln u+ k_2u^2+ k_1u+k_0, \quad i\in\{0,1,2\};\\
&\label{spich37}f^2=k_3u^m+ k_2u^2+ k_1u+k_0, \quad m\ne 0,1,2; \quad k_3\ne 0.
\end{eqnarray}
Analyzing of equation (\ref{spich13}) one can obtain the following results.
The general solution of that equation under functions (\ref{spich36}) with $k_3\ne 0$ coincides with (\ref{spich14}), what corresponds to the algebra $A^{\ker}$.

So, under $f^1=u$ one of the possible form of the function $f^2$ (up to redesignation of the constants $k_2$, $k_1$, $k_0$) is
\begin{eqnarray}\label{spich38}
f^2=k_1(u+k_2)^2+ k_3.
\end{eqnarray}
Further, if in function (\ref{spich37}) $m\ne 3$ and $k_2^2+k_1^2+k_0^2\ne 0$ then the general solution of equation
(\ref{spich13}) has ``trivial'' view (\ref{spich14}).
Consequently, another of the possible form of the function $f^2$ is
\begin{eqnarray}\label{spich39}
f^2=k_3u^m, \quad m\ne 0,1,2,3, \quad k_3\ne 0.
\end{eqnarray}
If in function (\ref{spich37}) $m=3$ then corresponding ``nontrivial'' solution of (\ref{spich13}) is only when
\begin{eqnarray}\label{spich40}
f^2=k_3u^3+3k_3k_2u^2+3k_3k_2^2u+k_3k_2^3=k_3(u+k_2)^3.
\end{eqnarray}
So, we combine:\\
functions $f^2$ (\ref{spich23}), (\ref{spich25}), corresponding the function $f^1=e^u$;\\
functions $f^2$ (\ref{spich27}), (\ref{spich28}), corresponding the function $f^1=\ln u$;\\
functions $f^2$ (\ref{spich31}), (\ref{spich33}), corresponding the function $f^1=u^m$ ($m\ne 0,1$);\\
functions $f^2$ (\ref{spich35}), (\ref{spich38}), (\ref{spich39}), (\ref{spich40}), corresponding the function $f^1=u$.\\
function $f^2=0$, corresponding an arbitrary function $f^1$ (see(\ref{spich20}), (\ref{spich21}) ).\\
As result we have the following possible forms of the functions $f^1$, $f^2$ given in the Table~\ref{tab:2} (columns two and tree).

Using the set of equivalence transformations (\ref{spich19}) which do not change the form of the functions $f^1$  from Table~\ref{tab:1}
we get simplified view $F^2$ of the obtained functions $f^2$ (see Table~\ref{tab:2},
column four).

\begin{table*}
\caption{\label{tab:2}Classification of functions $f^2$}
\begin{ruledtabular}
\begin{tabular}{lp{70pt}p{130pt}p{100pt}}
No. & $f^1(u)$ &$f^2(u)$ & $F^2(\bar {u})$ \\
\hline
1&$\forall$
&$0$ &$0$\\
2&$e^u$&$k_1e^u+k_2$, $k_1, k_2\ne 0$&$\varepsilon_1e^{\bar {u}}+\varepsilon_2$ \\
3&$e^u$ &$k_1e^{mu}$&$\tilde{\varepsilon}_1e^{m\bar {u}}$ \\
4&$\ln u$ &$k_1u^m$&$\tilde{\varepsilon}_1\bar {u}^m$ \\
5&$\ln u$ &$k_1u\ln u+k_2u$, $k_1\ne 0$&$\varepsilon_1\bar {u}\ln\bar {u} $ \\
6&$u^m$, $m\ne 0,1$ &$k_1u^n$&$\tilde{\varepsilon}_1\bar {u}^n$ \\
7&$u^m$, $m\ne 0,1$ &$k_1u^{m+1}+k_2u$, $k_1, k_2\ne 0$&$\varepsilon_1\bar {u}^{m+1}+\varepsilon_2\bar {u}$ \\
8&$u$ &$k_1e^{k_2u}$, $k_1, k_2\ne 0$&$\varepsilon_1e^{\bar {u}}$ \\
9&$u$ &$k_1u^m$, $k_1\ne 0$, $m\ne 0,1,2,3$&$\varepsilon_1\bar {u}^m$ \\
10&$u$ &$k_1(u+k_2)^3$, $k_1\ne 0$&$\varepsilon_1\bar {u}^3$ \\
11&$u$ &$k_1(u+k_2)^2+k_3$, $k_1\ne 0$&$\bar {u}^2+\tilde{\varepsilon}_1$ \\
12&$u$ &$k_1u+k_2$, $k_1\ne 0$&$\varepsilon_1\bar {u}$ \\
13&$u$ &$k_1$&$\tilde{\varepsilon}_1$ \\
\end{tabular}

Here $\varepsilon_i=\pm 1$, $\tilde{\varepsilon}_i\in \{-1,0,1\}$, $i\in \{1,2\}$.
\end{ruledtabular}
\end{table*}

\newpage
\begin{center}
{\bf Symmetry classification}
\end{center}

If an equation of the form (\ref{spich1}) admits Lie algebra $A^{\max }\supset A^{\ker}$, then it belongs to one
of subclasses listed in the Table~\ref{tab:2}. To get the standard group classification of the equation (\ref{spich1}) now
we need to find all possible Lie symmetries of each of those equations.
For this we substitute each pair of the functions $f^1(u)$, $f^2(u)=F^2(u)$ from the Table~\ref{tab:2} (columns two and four) to equations
(\ref{spich12}), (\ref{spich13}) under different values of $\varepsilon_i$, $\tilde{\varepsilon}$ and find general view of infinitesimal
operator (\ref{spich2}), (\ref{spich11}).

As result we obtain the following theorem.

\begin{theorem}\label{T5}
Any equation of the form \eqref{spich1}which is invariant under Lie algebra $A^{\max }\supset A^{\ker}$ is mapped
by an equivalence  transformations group $G^{\rm equiv}$ of the form \eqref{spich19} to one of the equations which are presented by
functions $f^1$ and $f^2$ in Table~\ref{tab:3} (columns two and three). The last column of the Table
presents generators of an algebra $A^{\max}$ that do not belong to the algebra $A^{\ker}$.
\end{theorem}

\begin{table*}
\caption{\label{tab:3}Group classification of equations \eqref{spich19} with respect to the $G^{\rm equiv}$}
\begin{ruledtabular}
\begin{flushleft}
\begin{tabular}{p{15pt}p{68pt}p{90pt}p{200pt}}
No. & $f^1(u)$ &$f^2(u)$ &Generators from $A^{\max}/A^{\ker}$ \\
\hline
1&$
\forall \ne {\rm const}
$
&$0$ &$t\partial_t +\frac{1}{2}x\partial_x +y\partial_y $\\
2&
$e^u$&
$\varepsilon_1 e^{mu},\  m\ne 0, 1$&
$t\partial_t +\frac{1}{2}x\partial_x +\frac{m-1}{m}y\partial_y -\frac{1}{m}\partial_u $\\
3&
$e^u$&
$\varepsilon_1 e^{u}+\varepsilon_2$&
$e^{\varepsilon_1 y}(\partial_y +\varepsilon_1\partial_u)$\\
4&
$e^u$&
$\varepsilon_1$&
$y\partial_y +\partial_u$\\
5&
$e^u$&
$\varepsilon_1 e^u$&
$t\partial_t +\frac{1}{2}x\partial_x -\partial_u,\quad e^{\varepsilon_1 y}(\partial_y +\varepsilon_1\partial_u)$\\
6&
$e^u$&
$0$&
$t\partial_t +\frac{1}{2}x\partial_x +y\partial_y,\quad y\partial_y +\partial_u$\\
7&
$\ln u$&
$\varepsilon_1 u^m, \ m\ne 1$&
$t\partial_t +\frac{1}{2}x\partial_x +(y-\frac{t}{m-1}\partial_y)-\frac{u}{m-1}\partial_u$\\
8&
$\ln u$&
$\varepsilon_1 u\ln u$&
$e^{\varepsilon_1 t}(\partial_y +\varepsilon_1u\partial_u)$\\
9&
$\ln u$&
$\varepsilon_1 u$&
$t\partial_t +\frac{1}{2}x\partial_x +(y+\frac{\varepsilon_1}{2}t^2\partial_y)+\varepsilon_1tu\partial_u,\quad t\partial_y$\\
10&
$\ln u$&
$0$&
$t\partial_t +\frac{1}{2}x\partial_x +y\partial_y,\quad t\partial_y+u\partial_u$\\
11&
$u^m \ m\ne 0,1$&
$\varepsilon_1u^n \ \ne m+1$&
$(n-1)t\partial_t +\frac{1}{2}(n-1)x\partial_x +(n-m-1)y\partial_y$\\
12&$u^m \ m\ne 0,1$&
$\varepsilon_1u^{m+1}+\varepsilon_2u$&
$e^{\varepsilon_1 my}(\partial_y +\varepsilon_1u\partial_u)$\\
13&
$u^m \ m\ne 0,1$&
$\varepsilon_1u^{m+1}$&
$t\partial_t +\frac{1}{2}x\partial_x -\frac{1}{m}u\partial_u,\quad e^{\varepsilon_1 my}(\partial_y +\varepsilon_1u\partial_u)$\\
14&$u^m \ m\ne 0,1$&
$0$&
$ t\partial_t +\frac{1}{2}x\partial_x +y\partial_y,\quad y\partial_y + \frac{1}{m}u\partial_u$\\
15&
$u$&
$\varepsilon_1e^u$&
$t\partial_t +\frac{1}{2}x\partial_x +(y-t)\partial_y-\partial_u$\\
16&
$u$&
$\varepsilon_1u^m \ m\ne 0,1,2$&
$t\partial_t +\frac{1}{2}x\partial_x +\frac{m-2}{m-1}y\partial_y-\frac{u}{m-1}\partial_u$\\
17&
$u$&
$u^2+1$&
$\begin{array}{l}
e^y(\sin t\partial_y+(\cos t+u\sin t)\partial_u),\\
e^y(\cos t\partial_y-(\sin t-u\cos t)\partial_u)
 \end{array}$\\
18&
$u$&
$u^2-1$&
$\begin{array}{l}
e^y(\sinh t\partial_y+(\cosh t+u\sinh t)\partial_u),\\
e^y(\cosh t\partial_y+(\sinh t+u\cosh t)\partial_u)
 \end{array}$\\
19&
$u$&
$\varepsilon_1u$&
$y\partial_y +u\partial_u,\quad e^{\varepsilon_1 t}(\partial_y +\varepsilon_1\partial_u)$\\
20&
$u$&
$u^2$&
$\begin{array}{l}
 t\partial_t +\frac{1}{2}x\partial_x -u\partial_u,\quad e^y(\partial_y +u\partial_u),\\
e^y(t\partial_y +(tu+1)\partial_u)
 \end{array}$\\
21&
$u$&
$\varepsilon_1$&
$\begin{array}{l}
 t\partial_t +\frac{1}{2}x\partial_x +\varepsilon_1t^2\partial_y+(2\varepsilon_1t-u)\partial_u,\\
 t\partial_y+\partial_u,\quad (y-\frac{1}{2}\varepsilon_1t^2)\partial_y+(u-\varepsilon_1t)\partial_u
 \end{array}$\\
22&
$u$&
$0$&
$t\partial_t +\frac{1}{2}x\partial_x +y\partial_y,\quad t\partial_y+\partial_u,\quad y\partial_y+u\partial_u$\\
\end{tabular}
\end{flushleft}

Here $\varepsilon_i=\pm 1$, $i\in \{1,2\}$.
\end{ruledtabular}
\end{table*}

\section{The additional equivalence transformations}
In this section we shall demonstrate the efficiency of using additional equivalence transformations for simplification of
nonlinear Kolmogorov equations arising in Table~\ref{tab:3} \cite{stog21, stog22, stog23, stog24}. The ways of finding additional equivalence
transformations is based on the fact that equivalent equations have the same dimensions of their invariance algebras $A^{\max}$.
To find additional equivalence transformations which reduced equation (\ref{spich1}) to another equation of same form use
conditions (\ref{spich16})--(\ref{spich18}). From~(\ref{spich17}) it is easy to derive that locally equivalent equations
from Table~\ref{tab:3} must have same function $f^1(u)$. So, it is easy to go through all the pairs of equations in Table~\ref{tab:3},
for which the dimensions of corresponding algebras and functions~$f^1(u)$ are the same and verify their equivalence.

Let us give some examples how to find such pairs of equivalent equations.
\begin{center}
{\bf Six-dimensional invariance algebras}
\end{center}

1) Let us consider the cases 21 and 22 of the equations from Table~\ref{tab:3}, which admit six-dimensional invariance algebras. Let us construct the
additional equivalence transformations reducing equation
\begin{eqnarray}\label{spich41}
u_t -u_{xx} +uu_y =\varepsilon_1,
\end{eqnarray}
to the equation
\begin{eqnarray}\label{spich42}
u_t -u_{xx} +uu_y =0.
\end{eqnarray}
Taking into account (\ref{spich16}) conditions (\ref{spich17}) and (\ref{spich18}) take the forms
\begin{eqnarray}\nonumber
Y_t+Y_yu=a_1^2(\theta u+\nu),\quad \theta_tu+\nu_t+(\theta_yu+\nu_y)u+\theta\varepsilon_1=0 .
\end{eqnarray}

Since the $\theta$ and $\nu$ do not depend on the variable $u$ we immediately arrive at
\begin{eqnarray*}
Y_t=a_1^2\nu,\quad Y_y=a_1^2\theta,\quad \theta_y=0,\quad \theta_t+\nu_y=0,\quad \nu_t+\theta\varepsilon_1=0 .
\end{eqnarray*}
Solutions of that system have the form
\begin{eqnarray}\label{spich43}
Y=C_1a_1^2y-\frac{C_1a_1^2\varepsilon_1}{2}t^2+C_2a_1^2t+C_3,\quad \theta=C_1,\quad \nu=-C_1\varepsilon_1t+C_2,
\end{eqnarray}
where $C_1$, $C_2$, $C_3$ are arbitrary constants, $C_1a_1\ne 0$.

Taking into account (\ref{spich43}) we get relevant additional equivalence transformations
\begin{eqnarray}\nonumber
\bar {t} = t,\quad \bar {x} = x,\quad \bar {y} = y-\frac{\varepsilon_1}{2}t^2,\quad \bar {u} = u-\varepsilon_1t,
\end{eqnarray}
which reducing equation (\ref{spich41}) to the equation (\ref{spich42}) (omitting the top line for new variables).

2) Similarly, the case 20 of the equation from Table~\ref{tab:3}
\begin{eqnarray}\label{spich44}
u_t -u_{xx} +uu_y =u^2
\end{eqnarray}
is reduced to the case 22 by the additional equivalence transformations
\begin{eqnarray*}
\bar {t} = t,\quad \bar {x} = x,\quad \bar {y} = -e^{-y},\quad \bar {u} = e^{-y}u.
\end{eqnarray*}

\begin{center}
{\bf Five-dimensional invariance algebras}
\end{center}
There are $9$ equations admitting five-dimensional invariance algebra. It turns out that only five among them are not locally equivalent.

3) We consider the cases 18 and 19 of the equations from Table~\ref{tab:3}, which admit five-dimensional invariance algebras. Let us construct the
additional equivalence transformations reducing equation
\begin{eqnarray}\label{spich45}
u_t -u_{xx} +uu_y =u^2-1
\end{eqnarray}
to the equation
\begin{eqnarray}\label{spich46}
u_t -u_{xx} +uu_y =u.
\end{eqnarray}
The corresponding conditions (\ref{spich17}) and (\ref{spich18})
\begin{eqnarray*}
Y_t+Y_yu=a_1^2(\theta u+\nu),\quad \quad \theta_tu+\nu_t+(\theta_yu+\nu_y)u+\theta(u^2-1)=a_1^2(\theta u+\nu)
\end{eqnarray*}
have a solution, which gives the following relevant additional equivalence transformations
\begin{eqnarray*}
\bar {t} = 2t,\quad \bar {x} = \sqrt{2}x,\quad \bar {y} = -2e^{t-y},\quad \bar {u} = e^{t-y}u-e^{t-y}.
\end{eqnarray*}
These transformations reduces equation (\ref{spich45}) to the equation (\ref{spich46}).

4) Let us consider the cases 17 and 19 of the equations from Table~\ref{tab:3}. It turns out that these equations are not equivalent. Assume the
contrary. Then, similarly to the previous case 3 we obtain the following equation (\ref{spich17}) and (\ref{spich18}) to search for equivalence
transformations:
\begin{eqnarray*}
Y_t+Y_yu=a_1^2(\theta u+\nu),\quad \theta_tu+\nu_t+(\theta_yu+\nu_y)u+\theta(u^2+1)=\varepsilon_1a_1^2(\theta u+\nu).
\end{eqnarray*}
Splitting the two equations in powers, we obtain the system of equations for functions $\theta$, $\nu$:
\begin{eqnarray*}
\theta_t=\nu_y,\quad \theta_y+\theta=0,\quad \theta_t+\nu_y=\varepsilon_1a_1^2\theta,\quad \nu_t+\theta=\varepsilon_1a_1^2\nu.
\end{eqnarray*}
Solving these equations consequentially, we arrive to the contradictory condition $a_1^4=-1$.

Acting in a similar way, applying conditions (\ref{spich17})--(\ref{spich18}), we have identified all others pairs of
equivalent equations from Table~\ref{tab:3} invariant under five- and four-dimensional algebras.
In the Table~\ref{tab:4} we give the functions $f^1$, $f^2$ and $F^1$, $F^2$ describing that pairs as well as
corresponding the additional equivalence transformations.

\begin{table*}
\caption{\label{tab:4}Additional equivalence transformations.}\vspace{3mm}
\begin{tabular}{p{20pt}p{55pt}p{60pt}p{160pt}p{50pt}p{40pt}}
\hline\\[-7mm]
\hline
no. & $f^1(u)$ &$f^2(u)$ & $\begin{array}{@{}c@{}}\text{Equivalence}\\[-3mm] \text{transformations}\end{array}$ &$F^1(\bar{u})$&$F^2(\bar{u})$ \\
\hline
1&$e^u$ &$\varepsilon_1e^u+\tilde\varepsilon_2$ &$\bar{t}=t$;\  $\bar{y}=-\varepsilon_1e^{-\varepsilon_1y}$;
 &$e^{\bar{u}}$&$\tilde\varepsilon_2$\\
  &&&$\bar{x}=x;$\  $\bar{u}=u-\varepsilon_1y$&&\\
2&$\ln u$& $\varepsilon_1u$&$\bar{t}=t$;\ $\bar{y}=y-\frac{\varepsilon_1}{2}t^2$;
 &$\ln \bar{u}$&$0$ \\
  &&&$\bar{x}=x$;\  $\bar{u}=e^{-\varepsilon_1t}u$&&\\
3&$u^m$,&
$\varepsilon_1u^{m+1}+$
 &$\bar{t}=t$;\  $\bar{y}=-\frac{\varepsilon_1}{m}e^{-\varepsilon_1my}$;
&$\bar{u}^m$&$\tilde\varepsilon_2\bar{u}$\\
 &$m\ne 0,1$&$\tilde\varepsilon_2u$& $\bar{x}=x$;\  $\bar{u}=e^{-\varepsilon_1y}u$&&\\
4&$u$ &$\varepsilon_1$&$\bar{t}=t$;\  $\bar{y}=y-\frac{\varepsilon_1}{2}t^2;$
 &$\bar{u}$&$0$\\
 &&&$\bar{x}=x;$\  $\bar{u}=u-\varepsilon_1$&&\\
5&$u$&$u^2$&$\bar{t}=t$;\  $\bar{y}=-e^{-y}$;
 &$\bar{u}$&$0$\\
  &&&$\bar{x}=x$;\  $\bar{u}=e^{-y}u$&&\\
6&$u$&$u^2-1$&$\bar{t}=2t$;\  $\bar{x}=\sqrt{2}x$; $\bar{y}=-2e^{t-y}$;
&$\bar{u}$&$\bar{u}$\\
&&&$\bar{u}=e^{t-y}u-e^{t-y}$&&\\
\hline\\[-7mm]
\hline
\end{tabular}

\vspace{2mm}
Here $\varepsilon_1=\pm 1$, $\tilde{\varepsilon}_2\in \{-1,0,1\}$.
\end{table*}

Finally, we exclude locally equivalent equations from Table~\ref{tab:3} using Table~\ref{tab:4} and formulate the theorem.

\begin{theorem}\label{T6}
All possible equations of the form \eqref{spich1} admitting Lie algebras $A^{\max }\supset A^{\ker}$ of symmetries
are reduced to one of the~$14$ ``canonical'' equations with functions given in Table~{\rm \ref{tab:5}} by an equivalence transformations
from $G^{\rm equiv}$ and the relevant additional
equivalence transformations from Table~{\rm \ref{tab:4}}. The maximal algebras of invariance of these ``canonical'' equations~\eqref{spich1} are presented in the fourth column of Table~{\rm \ref{tab:5}}.
\end{theorem}

\begin{table*}
\caption{\label{tab:5}Group classification of generalized nonlinear Kolmogorov type equations with respect to the additional
equivalence transformations which do not belong to $G^{\rm equiv}$}
\begin{ruledtabular}
\begin{flushleft}
\begin{tabular}{p{15pt}p{68pt}p{90pt}p{200pt}}
No. & $f^1(u)$ &$f^2(u)$ &Generators from $A^{\max}/A^{\ker}$ \\
\hline
1&$
\forall \ne {\rm const}
$
&$0$ &$t\partial_t +\frac{1}{2}x\partial_x +y\partial_y $\\
2&
$e^u$&
$\varepsilon_1 e^{mu},\  m\ne 0, 1$&
$t\partial_t +\frac{1}{2}x\partial_x +\frac{m-1}{m}y\partial_y -\frac{1}{m}\partial_u $\\
3&
$e^u$&
$\varepsilon_1$&
$y\partial_y +\partial_u$\\
4&
$e^u$&
$0$&
$t\partial_t +\frac{1}{2}x\partial_x +y\partial_y,\quad y\partial_y +\partial_u$\\
5&
$\ln u$&
$\varepsilon_1 u^m, \ m\ne 1$&
$t\partial_t +\frac{1}{2}x\partial_x +(y-\frac{t}{m-1}\partial_y)-\frac{u}{m-1}\partial_u$\\
6&
$\ln u$&
$\varepsilon_1 u\ln u$&
$e^{\varepsilon_1 t}(\partial_y +\varepsilon_1u\partial_u)$\\
7&
$\ln u$&
$0$&
$t\partial_t +\frac{1}{2}x\partial_x +y\partial_y,\quad t\partial_y+u\partial_u$\\
8&
$u^m \ m\ne 0,1$&
$\varepsilon_1u^n \ n\ne m+1$&
$(n-1)t\partial_t +\frac{1}{2}(n-1)x\partial_x +(n-m-1)y\partial_y$\\
9&$u^m \ m\ne 0,1$&
$0$&
$ t\partial_t +\frac{1}{2}x\partial_x +y\partial_y,\quad y\partial_y + \frac{1}{m}u\partial_u$\\
10&
$u$&
$\varepsilon_1e^u$&
$t\partial_t +\frac{1}{2}x\partial_x +(y-t)\partial_y-\partial_u$\\
11&
$u$&
$\varepsilon_1u^m \ m\ne 0,1,2$&
$t\partial_t +\frac{1}{2}x\partial_x +\frac{m-2}{m-1}y\partial_y-\frac{u}{m-1}\partial_u$\\
12&
$u$&
$u^2+1$&
$\begin{array}{l}
e^y(\sin t\partial_y+(\cos t+u\sin t)\partial_u),\\
e^y(\cos t\partial_y-(\sin t-u\cos t)\partial_u)
 \end{array}$\\
13&
$u$&
$\varepsilon_1u$&
$y\partial_y +u\partial_u,\quad e^{\varepsilon_1 t}(\partial_y +\varepsilon_1\partial_u)$\\
14&
$u$&
$0$&
$t\partial_t +\frac{1}{2}x\partial_x +y\partial_y,\quad t\partial_y+\partial_u,\quad y\partial_y+u\partial_u$\\
\end{tabular}
\end{flushleft}

Here $\varepsilon_1=\pm 1$.
\end{ruledtabular}
\end{table*}

\section{\label{6}Exact solutions}

\begin{center}
{\bf One-dimensional invariance subalgebras}
\end{center}
It is known that if a nonlinear differential equation with partial derivatives has nontrivial symmetry properties,
it allows us to use differential operators of the invariance algebra for its symmetry reduction \cite{stog3, stog2}.

We consider the use of the operators of the invariance algebra for symmetric reduction and the construction of
exact solutions on the example of the nonlinear Kolmogorov equation (\ref{spich1}) with the maximal six-dimensional invariance algebra.
In Table~\ref{tab:3}, we have three equations (cases 20, 21, 22) which have a six-dimensional invariance algebra.

Using the additional equivalence transformations from Table~\ref{tab:4}, we can reduce the cases 20 and 21 from Table~\ref{tab:3} to the case 22, i.e., we arrive at the equation
\begin{eqnarray}\label{spich47}
u_t -u_{xx} +uu_y =0.
\end{eqnarray}
Let us note that the equation (\ref{spich47}) are well investigated and that most of the exact solutions given below have been constructed before
\cite{stog18,stog15}.

Consider the equation
\begin{eqnarray}\label{spich48}
u_t -u_{xx} -uu_y =0,
\end{eqnarray}
which occurs in many problems of financial mathematics \cite{stog1}.\\
Applying the equivalence transformation
\begin{eqnarray*}
t\rightarrow t,\quad x\rightarrow x,\quad y\rightarrow -y,\quad u\rightarrow u,
\end{eqnarray*}
to equation (\ref{spich47}) and to the operators of the invariance algebra from 22 of Table~\ref{tab:3} we get the equation (\ref{spich48}),
which is invariant with respect to the six-dimensional invariance algebra whose operators have the form:
\begin{eqnarray*}
& P_0=\partial_t,\quad P_1=\partial_x,\quad P_2=\partial_y,\quad D_1=t\partial_t+\tfrac{1}{2}x\partial_x+y\partial_y,\\
&D_2=y\partial_y+u\partial_u,\quad G=t\partial_y-\partial_u .
\end{eqnarray*}
Consider the algebra with basis operators
\begin{eqnarray}\label{spich48a}
\begin{aligned}
&X_1=D_1-D_2=t\partial_t+\tfrac{1}{2}x\partial_x-u\partial_u,\\
&X_2=P_0=\partial_t,\quad X_3=D_2=y\partial_y+u\partial_u,\\
&X_4=P_2=\partial_y,\quad X_5=G=t\partial_y-\partial_u,\quad X_6=P_1=\partial_x,\\
\end{aligned}
\end{eqnarray}
which satisfy such non-zero commutation relations:
\begin{eqnarray*}
& [X_1,X_2]=-X_2,\quad [X_1,X_5]=X_5,\quad [X_1,X_6]=-\tfrac{1}{2}X_6,\\
&  [X_3,X_5]=-X_5,\quad [X_2,X_5]=X_4,\quad [X_3,X_4]=-X_4.
 \end{eqnarray*}

Further use of nontrivial symmetry properties of equation (\ref{spich48}) requires a classification of subalgebras of the algebra.
Using the well-known method of classifying subalgebras of a Lie algebra (see, for example, \cite{stog20,stog26}), we classify all
one-dimensional and two-dimensional Lie subalgebras of the algebra (\ref{spich48a}) up to the action
of transformations from the group of its inner automorphisms to which there corresponds a nontrivial reduction of equation (\ref{spich48})
to partial differential equations and ordinary differential equations, respectively. Thus, up to those action,
one-dimensional subalgebras are exhausted by such algebras:
\begin{eqnarray}\label{spich49}
\begin{aligned}
&\langle X_1\rangle,\quad \langle X_2\rangle,\quad \langle X_3\rangle,\quad \langle X_4\rangle,\quad \langle X_5\rangle,\quad
\langle X_6\rangle,\quad \langle X_1\pm X_4\rangle,\\
&\langle X_1+X_3\pm X_5\rangle,\quad \langle X_1+\alpha X_3\rangle,\quad \langle X_2\pm X_3\rangle,\quad \langle X_3+X_6\rangle,\\
&\langle X_2\pm X_3\pm X_6\rangle,\quad \langle X_2\pm X_5+\alpha X_6\rangle,\quad \langle X_2+X_6\rangle,\quad \langle X_2\pm X_5\rangle,\\
&\langle X_5+X_6\rangle,\quad \langle X_4+X_6\rangle,\quad (\alpha\ne 0).\\
\end{aligned}
\end{eqnarray}
Using these subalgebras, one can reduce equation (\ref{spich48}) to two-dimensional partial differential equations. For example, consider
the last operator from the list~(\ref{spich49})
\begin{eqnarray}\label{spich50}
\langle X_4+X_6\rangle=\partial_x+\partial_y.
\end{eqnarray}
Proceeding from the complete set of invariants of the operator (\ref{spich50})
\begin{eqnarray*}
I_1=t,\quad I_2=y-x,\quad I_3=u,
\end{eqnarray*}
we seek an invariant solution in the form
\begin{eqnarray}\label{spich51}
u=\varphi(\omega_1,\omega_2),\quad \omega_1=t,\quad \omega_2=y-x.
\end{eqnarray}
Substituting (\ref{spich51}) into equation (\ref{spich48}), we obtain the well-known Burgers equation
\begin{eqnarray*}
\varphi_{\omega_1}-\varphi_{\omega_2\omega_2}-\varphi\varphi_{\omega_2}=0 .
\end{eqnarray*}
In \cite{stog1}, the Cauchy problem for equation (\ref{spich48}) is examined
\begin{eqnarray}\label{spich52}
u_t -u_{xx} -uu_y =0,\quad u(0,x,y)=g(x,y).
\end{eqnarray}
If we require the solution of equation (\ref{spich48}) to be invariant with respect to the operator (\ref{spich50}),
and the function was of the form $g=g(\omega_2)$ then the solution of the Cauchy problem
\begin{eqnarray*}
\varphi_{\omega_1}-\varphi_{\omega_2\omega_2}-\varphi\varphi_{\omega_2}=0,\quad
\varphi({0,\omega_2})=g(\omega_2),
\end{eqnarray*}
has the form (see, for example, \cite{stog25})
\begin{eqnarray*}
 \varphi_({\omega_1,\omega_2})=2\frac{\partial}{\partial\omega_2}\ln G(\omega_1,\omega_2) ,
\end{eqnarray*}
where
\begin{eqnarray*}
 G(\omega_1,\omega_2)=\frac{1}{\sqrt{4\pi \omega_1}}\int_{-\infty}^\infty\exp
\left[-\frac{(\omega_2-\xi)^2}{4\omega_1}-\frac{1}{2}\int_0^\xi g(\tau)d\tau\right]d\xi ,
\end{eqnarray*}
and leads to the solution (\ref{spich51}) of the Cauchy problem (\ref{spich52}).

\begin{center}
{\bf Two-dimensional invariance subalgebras}
\end{center}
Let us employ the symmetries to reduce equation (\ref{spich48}) to various ordinary differential equations. For this
we classify all two-dimensional subalgebras of the algebra $L_6$ (\ref{spich48a}) up to the action of transformations of the group of its inner
automorphisms. Wherein $49$ two-dimensional subalgebras are obtained. Among the obtained two-dimensional subalgebras, the most interesting
from our point of view are the following, the corresponding ansatzes of which contain all the variables $t$, $x$, $y$:
\begin{eqnarray}\label{spich53}
\begin{aligned}
&\langle X_1,X_5\rangle,\quad \langle X_1,X_3\rangle,\quad \langle X_1\pm X_4,X_5\rangle,\quad \langle X_1+X_3,X_2+X_4\rangle,\\
&\langle X_1+2X_3,X_2+X_5\rangle,\quad \langle X_1+\alpha X_3,X_5\rangle,\quad \langle X_3,X_2+ X_6\rangle,\\
&\langle X_5,X_3+ X_6\rangle,\quad \langle X_3+X_6,X_2+\alpha X_6\rangle,\\
&\langle X_2+X_5+\beta X_6,X_4+\alpha X_6\rangle,\quad \langle X_2+X_6,X_4+\beta X_6\rangle,\\
&\langle X_2+X_5,X_4+\beta X_6\rangle,\quad \langle X_5,X_4+X_6\rangle,\quad \langle X_5+X_6,X_1+\tfrac{3}{2} X_3\rangle,\\
&\langle X_4+X_6,X_1+\frac{1}{2} X_3\rangle,\quad \langle X_5+X_6,X_4+\alpha X_6\rangle\quad (\alpha\ne 0,\quad \beta>0).\\
\end{aligned}
\end{eqnarray}

We have considered some cases of the two-dimensional subalgebras (\ref{spich53}) and obtained a symmetry reductions.
For some of them we construct exact invariant solutions of equation (\ref{spich48}). Let us give some examples.

\textbf{Example 1:}
$\langle X_4+\alpha X_6,X_5+X_6\rangle$.
The operators $X_4+\alpha X_6=\partial_y+\alpha\partial_x$, $X_5+X_6=\partial_x+t\partial_y-\partial_u$
produces the similarity solutions
\begin{eqnarray*}
 u=\frac{x-\alpha y}{\alpha t-1}+\varphi(\omega),\quad \omega=t,
\end{eqnarray*}
that reduces (\ref{spich48}) to the ordinary differential equations $\varphi'(\omega)=0$.

Solving the corresponding reduced equation, we obtain an invariant solution
\begin{eqnarray}\label{spich54}
 u=\frac{x-\alpha y}{\alpha t-1}+C,
\end{eqnarray}
where $C$ is an arbitrary constant.

In \cite{stog1} the exact solution of equation (\ref{spich48})
\begin{eqnarray}\label{spich55}
 u=\frac{x+y}{1-t},
\end{eqnarray}
is considered. The solution (\ref{spich55}) can be obtained from (\ref{spich54}) with $C=0$, $\alpha=1$
and using the equivalence transformation $t\rightarrow t$, $x\rightarrow x$, $y\rightarrow -y$, $u\rightarrow -u$.

\textbf{Example 2:} $\langle X_3+X_6,X_5\rangle$.
We obtain the similarity reduction
\begin{eqnarray*}
 u=-\frac{y}{t}+e^xf(\omega),\quad \omega=t,
\end{eqnarray*}
that reduces (\ref{spich48}) to the differential equation
\begin{eqnarray*}
\omega f'+(1-\omega)f=0.
\end{eqnarray*}
Solving the latter ordinary differential equation we obtain the similarity solution of equation (\ref{spich48})
\begin{eqnarray*}
 u=\frac{Ce^{x+t}-y}{t},\quad C ={\rm const}.
\end{eqnarray*}

\textbf{Example 3:} $\langle X_4+X_6,X_5\rangle$.
We obtain the similarity reduction
\begin{eqnarray*}
 u=\frac{x-y}{t}+f(\omega),\quad \omega=t;\quad \omega f'+f=0.
\end{eqnarray*}
The reduced equations obtained obviously have exact solutions which lead to the above exact solution of equations (\ref{spich48})
\begin{eqnarray*}
 u=\frac{x-y+C}{t}.
\end{eqnarray*}

\textbf{Example 4:} $\langle X_2+X_6,X_4+\beta X_6\rangle$, $\beta>0$.
We obtain the similarity reduction
\begin{eqnarray*}
 u=f(\omega),\quad \omega=\beta y+t-x;\quad f''+(\beta f-1)f'=0.
\end{eqnarray*}
Substituting the solution of the reduced equation into the ansatz, we construct the exact solutions of equation (\ref{spich48}):
\begin{eqnarray*}
 &u=\frac{2}{\beta(\beta y+t-x+C_1)}+\frac{1}{\beta};\\
&u=\frac{\sqrt{-1-2\beta C}}{\beta}+\tan\left(-\frac{\sqrt{-1-2\beta C}}{\beta}(\beta y+t-x)+C_1\right)+\frac{1}{\beta},\quad
2\beta C+1<0;\\
&u=\frac{-2\sqrt{1+2\beta C}}{\beta(C_1e^{-\sqrt{1+2\beta C)}(\beta y+t-x)}-1)}+\frac{1-\sqrt{1+2\beta C}}{\beta},\quad 2\beta C+1>0,\quad
C, C_1 ={\rm const}.
\end{eqnarray*}

\textbf{Example 5:} $\langle X_3,X_2+X_6\rangle$.
We obtain the similarity reduction
\begin{eqnarray*}
 u=yf(\omega),\quad \omega=t-x;\quad f''-f'+f^2=0,
\end{eqnarray*}
which in turn reduces to the Emden--Fowler equation, whose solutions are known for a number of parameter values \cite{stog27,stog29}.

\section{\label{6}Conclusions}

In this paper, the group classification in the class of nonlinear Kolmogorov equations (\ref{spich1}) is presented completely.
The main results on classification are collected in Table~\ref{tab:3}, where we list inequivalent cases of extensions with the corresponding
Lie invariance algebras. The list contains 22 equations and cannot be shortened by local substitution from group of equivalence transformations.
In Table~\ref{tab:4} we write down all the additional equivalence transformations, reducing some equations from our classification to others of simpler forms.
We have proved that exactly 14 equations in Table~\ref{tab:5} are reduced to other equations listed by the additional equivalence transformations.
For equations which have the maximal symmetry properties six-dimensional invariance algebra we can reduce these equations to equation
(\ref{spich48}) using the set additional equivalence transformations. Using the well-known method of classifying subalgebras of a Lie algebra,
we classify all one-dimensional and two-dimensional Lie subalgebras of an algebra which corresponds a nontrivial reduction of equation
(\ref{spich48}) to partial differential equations and ordinary differential equations respectively.


\begin{thebibliography}{27}%
\makeatletter
\providecommand \@ifxundefined [1]{%
 \@ifx{#1\undefined}
}%
\providecommand \@ifnum [1]{%
 \ifnum #1\expandafter \@firstoftwo
 \else \expandafter \@secondoftwo
 \fi
}%
\providecommand \@ifx [1]{%
 \ifx #1\expandafter \@firstoftwo
 \else \expandafter \@secondoftwo
 \fi
}%
\providecommand \natexlab [1]{#1}%
\providecommand \enquote  [1]{``#1''}%
\providecommand \bibnamefont  [1]{#1}%
\providecommand \bibfnamefont [1]{#1}%
\providecommand \citenamefont [1]{#1}%
\providecommand \href@noop [0]{\@secondoftwo}%
\providecommand \href [0]{\begingroup \@sanitize@url \@href}%
\providecommand \@href[1]{\@@startlink{#1}\@@href}%
\providecommand \@@href[1]{\endgroup#1\@@endlink}%
\providecommand \@sanitize@url [0]{\catcode `\\12\catcode `\$12\catcode
  `\&12\catcode `\#12\catcode `\^12\catcode `\_12\catcode `\%12\relax}%
\providecommand \@@startlink[1]{}%
\providecommand \@@endlink[0]{}%
\providecommand \url  [0]{\begingroup\@sanitize@url \@url }%
\providecommand \@url [1]{\endgroup\@href {#1}{\urlprefix }}%
\providecommand \urlprefix  [0]{URL }%
\providecommand \Eprint [0]{\href }%
\providecommand \doibase [0]{http://dx.doi.org/}%
\providecommand \selectlanguage [0]{\@gobble}%
\providecommand \bibinfo  [0]{\@secondoftwo}%
\providecommand \bibfield  [0]{\@secondoftwo}%
\providecommand \translation [1]{[#1]}%
\providecommand \BibitemOpen [0]{}%
\providecommand \bibitemStop [0]{}%
\providecommand \bibitemNoStop [0]{.\EOS\space}%
\providecommand \EOS [0]{\spacefactor3000\relax}%
\providecommand \BibitemShut  [1]{\csname bibitem#1\endcsname}%
\let\auto@bib@innerbib\@empty
\bibitem [{\citenamefont {Akhatov}, \citenamefont {Gazizov},\ and\
  \citenamefont {Ibragimov}(1991)}]{stog19}%
  \BibitemOpen
  \bibfield  {author} {\bibinfo {author} {\bibnamefont {Akhatov}, \bibfnamefont
  {I.~S.}}, \bibinfo {author} {\bibnamefont {Gazizov}, \bibfnamefont {R.~K.}},
  \ and\ \bibinfo {author} {\bibnamefont {Ibragimov}, \bibfnamefont {N.~K.}},\
  }\bibfield  {title} {\enquote {\bibinfo {title} {Nonlocal symmetries.
  heuristic approach},}\ }\href@noop {} {\bibfield  {journal} {\bibinfo
  {journal} {J. Soviet Mathematics}\ }\textbf {\bibinfo {volume} {55(1)}},\
  \bibinfo {pages} {1401–--50} (\bibinfo {year} {1991})}\BibitemShut
  {NoStop}%
\bibitem [{\citenamefont {Bluman}\ and\ \citenamefont {Kumei}(1989)}]{stog4}%
  \BibitemOpen
  \bibfield  {author} {\bibinfo {author} {\bibnamefont {Bluman}, \bibfnamefont
  {G.~W.}}\ and\ \bibinfo {author} {\bibnamefont {Kumei}, \bibfnamefont {S.}},\
  }\href@noop {} {\emph {\bibinfo {title} {Symmetries and Differential
  Equations}}}\ (\bibinfo  {publisher} {Springer-Verlag},\ \bibinfo {year}
  {1989})\BibitemShut {NoStop}%
\bibitem [{\citenamefont {Cherniha}, \citenamefont {Serov},\ and\ \citenamefont
  {Rassokha}(2008)}]{stog23}%
  \BibitemOpen
  \bibfield  {author} {\bibinfo {author} {\bibnamefont {Cherniha},
  \bibfnamefont {R.}}, \bibinfo {author} {\bibnamefont {Serov}, \bibfnamefont
  {M.}}, \ and\ \bibinfo {author} {\bibnamefont {Rassokha}, \bibfnamefont
  {I.}},\ }\bibfield  {title} {\enquote {\bibinfo {title} {Lie symmetries and
  form-preserving transformations of reaction diffusion convection
  equations},}\ }\href@noop {} {\bibfield  {journal} {\bibinfo  {journal} {J.
  Math. Anal. Appl.}\ }\textbf {\bibinfo {volume} {342}},\ \bibinfo {pages}
  {1363--79} (\bibinfo {year} {2008})}\BibitemShut {NoStop}%
\bibitem [{\citenamefont {Demetriou}, \citenamefont {Christou},\ and\
  \citenamefont {Sophocleous}(2007)}]{stog18}%
  \BibitemOpen
  \bibfield  {author} {\bibinfo {author} {\bibnamefont {Demetriou},
  \bibfnamefont {E.}}, \bibinfo {author} {\bibnamefont {Christou},
  \bibfnamefont {M.~A.}}, \ and\ \bibinfo {author} {\bibnamefont {Sophocleous},
  \bibfnamefont {C.}},\ }\bibfield  {title} {\enquote {\bibinfo {title} {On the
  classification of similarity solutions of a two-dimensional
  diffusion-advection equation},}\ }\href@noop {} {\bibfield  {journal}
  {\bibinfo  {journal} {J. Appl. Math. Comp.}\ }\textbf {\bibinfo {volume}
  {187}},\ \bibinfo {pages} {1333--50} (\bibinfo {year} {2007})}\BibitemShut
  {NoStop}%
\bibitem [{\citenamefont {Demetriou}, \citenamefont {Ivanova},\ and\
  \citenamefont {Sophocleous}(2008)}]{stog17}%
  \BibitemOpen
  \bibfield  {author} {\bibinfo {author} {\bibnamefont {Demetriou},
  \bibfnamefont {E.}}, \bibinfo {author} {\bibnamefont {Ivanova}, \bibfnamefont
  {N.}}, \ and\ \bibinfo {author} {\bibnamefont {Sophocleous}, \bibfnamefont
  {C.}},\ }\bibfield  {title} {\enquote {\bibinfo {title} {Group analysis of
  (2+1)- and (3+1)-dimensional diffusion-convection equations},}\ }\href@noop
  {} {\bibfield  {journal} {\bibinfo  {journal} {J. Math. Anal. Appl.}\
  }\textbf {\bibinfo {volume} {348}},\ \bibinfo {pages} {55--65} (\bibinfo
  {year} {2008})}\BibitemShut {NoStop}%
\bibitem [{\citenamefont {Dorodnitsyn}(1982)}]{stog28}%
  \BibitemOpen
  \bibfield  {author} {\bibinfo {author} {\bibnamefont {Dorodnitsyn},
  \bibfnamefont {V.~A.}},\ }\bibfield  {title} {\enquote {\bibinfo {title} {On
  invariant solutions of the equation of non-linear heat conduction with a
  source},}\ }\href@noop {} {\bibfield  {journal} {\bibinfo  {journal}
  {U.S.S.R. Comput. Math. Math. Phys.}\ }\textbf {\bibinfo {volume} {22}},\
  \bibinfo {pages} {115--22} (\bibinfo {year} {1982})}\BibitemShut {NoStop}%
\bibitem [{\citenamefont {Dorodnitsyn}, \citenamefont {Knyazeva},\ and\
  \citenamefont {Svirschevsky}(1983)}]{stog8}%
  \BibitemOpen
  \bibfield  {author} {\bibinfo {author} {\bibnamefont {Dorodnitsyn},
  \bibfnamefont {V.~A.}}, \bibinfo {author} {\bibnamefont {Knyazeva},
  \bibfnamefont {I.~V.}}, \ and\ \bibinfo {author} {\bibnamefont
  {Svirschevsky}, \bibfnamefont {S.~R.}},\ }\bibfield  {title} {\enquote
  {\bibinfo {title} {Group properties of the heat equation with a source in
  two-dimensional and three-dimensional cases},}\ }\href@noop {} {\bibfield
  {journal} {\bibinfo  {journal} {Diferential equations}\ }\textbf {\bibinfo
  {volume} {19}},\ \bibinfo {pages} {1215--23} (\bibinfo {year}
  {1983})}\BibitemShut {NoStop}%
\bibitem [{\citenamefont {Elwakil}, \citenamefont {Zahran},\ and\ \citenamefont
  {Sabry}(2005)}]{stog15}%
  \BibitemOpen
  \bibfield  {author} {\bibinfo {author} {\bibnamefont {Elwakil}, \bibfnamefont
  {S.~A.}}, \bibinfo {author} {\bibnamefont {Zahran}, \bibfnamefont {M.}}, \
  and\ \bibinfo {author} {\bibnamefont {Sabry}, \bibfnamefont {R.}},\
  }\bibfield  {title} {\enquote {\bibinfo {title} {Group classification and
  symmetry reduction of a (2+1) dimensional diffusion-advection equation},}\
  }\href@noop {} {\bibfield  {journal} {\bibinfo  {journal} {J. Applied Math.
  Phys.}\ }\textbf {\bibinfo {volume} {56}},\ \bibinfo {pages} {986--99}
  (\bibinfo {year} {2005})}\BibitemShut {NoStop}%
\bibitem [{\citenamefont {Huang}\ and\ \citenamefont {Ivanova}(2016)}]{stog11}%
  \BibitemOpen
  \bibfield  {author} {\bibinfo {author} {\bibnamefont {Huang}, \bibfnamefont
  {D.}}\ and\ \bibinfo {author} {\bibnamefont {Ivanova}, \bibfnamefont
  {N.~M.}},\ }\bibfield  {title} {\enquote {\bibinfo {title} {Algorithmic
  framework for group analysis of differential equations and its application to
  generalized zakharov-kuznetsov equations},}\ }\href@noop {} {\bibfield
  {journal} {\bibinfo  {journal} {J. Differential Equations}\ }\textbf
  {\bibinfo {volume} {260}},\ \bibinfo {pages} {2354--82} (\bibinfo {year}
  {2016})}\BibitemShut {NoStop}%
\bibitem [{\citenamefont {Ibragimov}(1994)}]{stog7}%
  \BibitemOpen
  \bibinfo {editor} {\bibnamefont {Ibragimov}, \bibfnamefont {N.~H.}},\ ed.,\
  \enquote {\bibinfo {title} {Crc handbook of {L}ie group analysis of
  differential equations},}\ \ (\bibinfo  {publisher} {CRC Press},\ \bibinfo
  {year} {1994})\BibitemShut {NoStop}%
\bibitem [{\citenamefont {Ibragimov}(1995)}]{stog5}%
  \BibitemOpen
  \bibinfo {editor} {\bibnamefont {Ibragimov}, \bibfnamefont {N.~H.}},\ ed.,\
  \enquote {\bibinfo {title} {Crc handbook of {L}ie group analysis of
  differential equations},}\ \ (\bibinfo  {publisher} {CRC Press},\ \bibinfo
  {year} {1995})\ pp.\ \bibinfo {pages} {291--328}\BibitemShut {NoStop}%
\bibitem [{\citenamefont {Kamke}(1977)}]{stog27}%
  \BibitemOpen
  \bibinfo {editor} {\bibnamefont {Kamke}, \bibfnamefont {E.}},\ ed.,\ \enquote
  {\bibinfo {title} {Differentialgleichungen: L{\"o}sungsmethoden und
  l{\"o}sungen, {I}, gew{\"o}hnliche differentialgleichungen},}\ \ (\bibinfo
  {publisher} {B. G. Teubner, Leipzig},\ \bibinfo {year} {1977})\BibitemShut
  {NoStop}%
\bibitem [{\citenamefont {Kingston}\ and\ \citenamefont
  {Sophocleous}(1998)}]{stog21}%
  \BibitemOpen
  \bibfield  {author} {\bibinfo {author} {\bibnamefont {Kingston},
  \bibfnamefont {J.~G.}}\ and\ \bibinfo {author} {\bibnamefont {Sophocleous},
  \bibfnamefont {C.}},\ }\bibfield  {title} {\enquote {\bibinfo {title} {On
  form-preserving point transformations of partial differential equations},}\
  }\href@noop {} {\bibfield  {journal} {\bibinfo  {journal} {J. Phys. A: Math.
  Gen.}\ }\textbf {\bibinfo {volume} {31}},\ \bibinfo {pages} {1597--619}
  (\bibinfo {year} {1998})}\BibitemShut {NoStop}%
\bibitem [{\citenamefont {Lahno}, \citenamefont {Spichak},\ and\ \citenamefont
  {Stogniy}(nian)}]{stog20}%
  \BibitemOpen
  \bibinfo {editor} {\bibnamefont {Lahno}, \bibfnamefont {V.~I.}}, \bibinfo
  {editor} {\bibnamefont {Spichak}, \bibfnamefont {S.~V.}}, \ and\ \bibinfo
  {editor} {\bibnamefont {Stogniy}, \bibfnamefont {V.~I.}},\ eds.,\ \enquote
  {\bibinfo {title} {{S}ymmetry {A}nalysis of {E}volution {T}ype
  {E}quations},}\ \ (\bibinfo  {publisher} {Kyiv; Institute of Mathematics of
  NAS of Ukraine},\ \bibinfo {year} {2002 (in Ukrainian)})\BibitemShut
  {NoStop}%
\bibitem [{\citenamefont {Nikititin}\ and\ \citenamefont
  {Popovych}(2001)}]{stog12}%
  \BibitemOpen
  \bibfield  {author} {\bibinfo {author} {\bibnamefont {Nikititin},
  \bibfnamefont {A.~G.}}\ and\ \bibinfo {author} {\bibnamefont {Popovych},
  \bibfnamefont {R.}},\ }\bibfield  {title} {\enquote {\bibinfo {title} {Group
  classification nonlinear {S}chrodinger equations},}\ }\href@noop {}
  {\bibfield  {journal} {\bibinfo  {journal} {Ukr. Math. J.}\ }\textbf
  {\bibinfo {volume} {53}},\ \bibinfo {pages} {1255--65} (\bibinfo {year}
  {2001})}\BibitemShut {NoStop}%
\bibitem [{\citenamefont {Olver}(1986)}]{stog3}%
  \BibitemOpen
  \bibfield  {author} {\bibinfo {author} {\bibnamefont {Olver}, \bibfnamefont
  {P.~J.}},\ }\href@noop {} {\emph {\bibinfo {title} {Application of Lie Groups
  to Differential Equations}}}\ (\bibinfo  {publisher} {Springer},\ \bibinfo
  {year} {1986})\BibitemShut {NoStop}%
\bibitem [{\citenamefont {Ovsiannikov}(1959)}]{stog6}%
  \BibitemOpen
  \bibfield  {author} {\bibinfo {author} {\bibnamefont {Ovsiannikov},
  \bibfnamefont {L.~V.}},\ }\bibfield  {title} {\enquote {\bibinfo {title}
  {Group properties of nonlinear heat equation},}\ }\href@noop {} {\bibfield
  {journal} {\bibinfo  {journal} {Dokl. AN SSSR}\ }\textbf {\bibinfo {volume}
  {125}},\ \bibinfo {pages} {492--5} (\bibinfo {year} {1959})}\BibitemShut
  {NoStop}%
\bibitem [{\citenamefont {Ovsiannikov}(1982)}]{stog2}%
  \BibitemOpen
  \bibfield  {author} {\bibinfo {author} {\bibnamefont {Ovsiannikov},
  \bibfnamefont {L.~V.}},\ }\href@noop {} {\emph {\bibinfo {title} {Group
  Analysis of Differential Equations}}}\ (\bibinfo  {publisher} {Acad. Press},\
  \bibinfo {year} {1982})\BibitemShut {NoStop}%
\bibitem [{\citenamefont {Pascucci}\ and\ \citenamefont
  {Polidoro}(2003)}]{stog1}%
  \BibitemOpen
  \bibfield  {author} {\bibinfo {author} {\bibnamefont {Pascucci},
  \bibfnamefont {A.}}\ and\ \bibinfo {author} {\bibnamefont {Polidoro},
  \bibfnamefont {S.}},\ }\bibfield  {title} {\enquote {\bibinfo {title} {On the
  {C}auchy problem for a nonlinear {K}olmogorov equation},}\ }\href@noop {}
  {\bibfield  {journal} {\bibinfo  {journal} {SIAM J. Math. Anal.}\ }\textbf
  {\bibinfo {volume} {35}},\ \bibinfo {pages} {579--95} (\bibinfo {year}
  {2003})}\BibitemShut {NoStop}%
\bibitem [{\citenamefont {Patera}\ and\ \citenamefont
  {Winternitz}(1975)}]{stog26}%
  \BibitemOpen
  \bibfield  {author} {\bibinfo {author} {\bibnamefont {Patera}, \bibfnamefont
  {J.}}\ and\ \bibinfo {author} {\bibnamefont {Winternitz}, \bibfnamefont
  {P.}},\ }\bibfield  {title} {\enquote {\bibinfo {title} {Continuous subgroups
  of the fundamental groups of physics. {I}. general methods and the poincare
  group},}\ }\href@noop {} {\bibfield  {journal} {\bibinfo  {journal} {J. Math.
  Phys.}\ }\textbf {\bibinfo {volume} {16}},\ \bibinfo {pages} {1597--1624}
  (\bibinfo {year} {1975})}\BibitemShut {NoStop}%
\bibitem [{\citenamefont {Polyanin}\ and\ \citenamefont
  {Zaitsev}(2003)}]{stog29}%
  \BibitemOpen
  \bibinfo {editor} {\bibnamefont {Polyanin}, \bibfnamefont {A.~D.}}\ and\
  \bibinfo {editor} {\bibnamefont {Zaitsev}, \bibfnamefont {V.~F.}},\ eds.,\
  \enquote {\bibinfo {title} {Handbook of exact solutions for ordinary
  differential equations},}\ \ (\bibinfo  {publisher} {Chapman \& Hall/CRC,
  Boca Raton, FL},\ \bibinfo {year} {2003})\BibitemShut {NoStop}%
\bibitem [{\citenamefont {Polyanin}\ and\ \citenamefont
  {Zaitsev}(2004)}]{stog25}%
  \BibitemOpen
  \bibinfo {editor} {\bibnamefont {Polyanin}, \bibfnamefont {A.~D.}}\ and\
  \bibinfo {editor} {\bibnamefont {Zaitsev}, \bibfnamefont {V.~F.}},\ eds.,\
  \enquote {\bibinfo {title} {Handbook of nonlinear partial differential
  equations},}\ \ (\bibinfo  {publisher} {Chapman \& Hall/CRC, Boca Raton,
  FL},\ \bibinfo {year} {2004})\BibitemShut {NoStop}%
\bibitem [{\citenamefont {Popovych}\ and\ \citenamefont
  {Ivanova}(2004)}]{stog22}%
  \BibitemOpen
  \bibfield  {author} {\bibinfo {author} {\bibnamefont {Popovych},
  \bibfnamefont {R.~O.}}\ and\ \bibinfo {author} {\bibnamefont {Ivanova},
  \bibfnamefont {N.~M.}},\ }\bibfield  {title} {\enquote {\bibinfo {title} {New
  results on group classification of nonlinear diffusion-convection
  equations},}\ }\href@noop {} {\bibfield  {journal} {\bibinfo  {journal} {J.
  Phys. A: Math. Gen.}\ }\textbf {\bibinfo {volume} {37}},\ \bibinfo {pages}
  {7547--65} (\bibinfo {year} {2004})}\BibitemShut {NoStop}%
\bibitem [{\citenamefont {Shtelen}\ and\ \citenamefont {Stogny}(1989)}]{stog9}%
  \BibitemOpen
  \bibfield  {author} {\bibinfo {author} {\bibnamefont {Shtelen}, \bibfnamefont
  {W.~M.}}\ and\ \bibinfo {author} {\bibnamefont {Stogny}, \bibfnamefont
  {V.~I.}},\ }\bibfield  {title} {\enquote {\bibinfo {title} {Symmetry
  properties of one- and two-dimensional {F}okker-{P}lanck equations},}\
  }\href@noop {} {\bibfield  {journal} {\bibinfo  {journal} {J. Phys. A: Math.
  Gen.}\ }\textbf {\bibinfo {volume} {22}},\ \bibinfo {pages} {539--43}
  (\bibinfo {year} {1989})}\BibitemShut {NoStop}%
\bibitem [{\citenamefont {Spichak}\ and\ \citenamefont
  {Stogniy}(1998)}]{stog10}%
  \BibitemOpen
  \bibfield  {author} {\bibinfo {author} {\bibnamefont {Spichak}, \bibfnamefont
  {S.~V.}}\ and\ \bibinfo {author} {\bibnamefont {Stogniy}, \bibfnamefont
  {V.~I.}},\ }\bibfield  {title} {\enquote {\bibinfo {title} {Symmetry
  classification and exact solutions of the {K}ramers equation},}\ }\href@noop
  {} {\bibfield  {journal} {\bibinfo  {journal} {J. Math. Phys.}\ }\textbf
  {\bibinfo {volume} {39}},\ \bibinfo {pages} {3505--10} (\bibinfo {year}
  {1998})}\BibitemShut {NoStop}%
\bibitem [{\citenamefont {Vasilenko}\ and\ \citenamefont
  {Yehorchenko}(2001)}]{stog13}%
  \BibitemOpen
  \bibfield  {author} {\bibinfo {author} {\bibnamefont {Vasilenko},
  \bibfnamefont {O.~F.}}\ and\ \bibinfo {author} {\bibnamefont {Yehorchenko},
  \bibfnamefont {I.~A.}},\ }\bibfield  {title} {\enquote {\bibinfo {title}
  {Group classification of multidimensional nonlinear wave equations},}\
  }\href@noop {} {\bibfield  {journal} {\bibinfo  {journal} {Proc. Inst.
  Mathematics of NAS of Ukraine}\ }\textbf {\bibinfo {volume} {36}},\ \bibinfo
  {pages} {63} (\bibinfo {year} {2001})}\BibitemShut {NoStop}%
\bibitem [{\citenamefont {Winternitz}\ and\ \citenamefont
  {Gazeau}(1992)}]{stog24}%
  \BibitemOpen
  \bibfield  {author} {\bibinfo {author} {\bibnamefont {Winternitz},
  \bibfnamefont {P.}}\ and\ \bibinfo {author} {\bibnamefont {Gazeau},
  \bibfnamefont {J.}},\ }\bibfield  {title} {\enquote {\bibinfo {title}
  {Allowed transformations and symmetry classes of variable coefficient
  {K}orteweg-de {V}ries equations},}\ }\href@noop {} {\bibfield  {journal}
  {\bibinfo  {journal} {Phys. Lett. A}\ }\textbf {\bibinfo {volume} {167}},\
  \bibinfo {pages} {246--50} (\bibinfo {year} {1992})}\BibitemShut {NoStop}%
\end{thebibliography}

\providecommand{\noopsort}[1]{}\providecommand{\singleletter}[1]{#1}%

\end{document}